\documentclass[smallextended]{svjour3}       
\smartqed  
\usepackage{mathptmx}      
\usepackage[dvips]{graphicx,psfrag}
\usepackage{amsmath,amssymb}
\usepackage{color}
\journalname{Japan J.\ Indust.\ Appl.\ Math.}
\newcommand{\Ref}[1]{{\rm (\ref{#1})}}

\newcommand{\R}{\mathbb R}

\newcommand{\dd}{\mathrm{d}}

\newcommand{\bfx}{{\bf x}}

\newcommand{\hK}{\widehat{K}}
\newcommand{\Fab}{F_{\alpha,\beta}}
\newcommand{\Kab}{K_{\alpha,\beta}}

\newcommand{\Kst}{K_{s,t}}

\newcommand{\KK}{K_{s,t}^\eta}

\newcommand{\T}{\mathcal{T}}

\newcommand{\Ih}{{I}_h}
\newcommand{\edit}[1]{{\color{black}#1}}
\begin{document}

\title{A Babu\v{s}ka-Aziz type proof of the circumradius condition
}


\author{Kenta Kobayashi \and Takuya Tsuchiya}


\institute{Kenta Kobayashi \at
           Graduate School of Commerce and Management, \\
         Hitotsubashi University, Japan \\
           \email{kenta.k@r.hit-u.ac.jp} \and
           Takuya Tsuchiya \at
           Graduate School of Science and Engineering, \\
           Ehime Unive\edit{r}sity, Japan \\
           \email{tsuchiya@math.sci.ehime-u.ac.jp}
}

\date{Received: date / Accepted: date}

\maketitle

\begin{abstract}
In this paper the error of polynomial interpolation of degree $1$
on triangles is considered.  The circumradius condition, which is
more general than the maximum angle condition, is explained and
proved by the technique given by Babu\v{s}ka-Aziz.
\keywords{interpolation error \and \edit{finite element methods} \and
\edit{the maximum angle condition} \and the circumradius condition}
\subclass{65D05, 65N30}
\end{abstract}

\section{Introduction --- the circumradius condition}
\label{intro}
Let $\mathcal{P}_1$ be the set of polynomials whose degree are
at most $1$.  Let $K \subset \R^2$ be any triangle with apexes
$\bfx_i$ $i=1,2,3$.  Then, for a function $v \in W^{2,p}(K)$ the
$\mathcal{P}_1$ interpolation $\Ih v$ on $K$ is defined by
$(\Ih v)(\bfx_i) = v(\bfx_i)$.  Note that the interpolation $\Ih v$
is well-defined for $v \in W^{2,p}(K)$ since $W^{2,p}(K)$ is imbedded
to $C(\overline{K})$ for any $p \in [1,\infty]$.\footnote{For
the critical imbedding $W^{2,1}(K) \subset C(\overline{K})$, see
\cite[p.300]{KJF}.}
 Analyzing the error
\[
    \|v - \Ih v\|_{1,p,K}
\]
is \edit{particularly} important for the error analysis of
finite element methods. 
\edit{There is a long history of research into
this error bound}.  \edit{We present some well-known results}.
Let $h_K$ be the diameter (or the length of the longest edge) of $K$,
and $\rho_K$ be the radius of the inscribed circle of $K$. 
\begin{itemize}
\item \textbf{The minimum angle condition}, Zl\'amal \cite{Z} (1968). \\
{\it Let $\theta_0$, $0 < \theta_0 < \pi/3$ be a constant. If
any angle $\theta$ of $K$ satisfies $\theta \ge \theta_0$
and $h_K \le 1$, then there exists a constant $C = C(\theta_0)$
independent of $h_K$ such that}
\[
   \|v - \Ih v\|_{1,2,K} \le C h_K |v|_{2,2,K}, \qquad
   \forall v \in H^2(K).
\]
 \item \textbf{The regularity condition},
 see, for example, Ciarlet \cite{C}. \\
{\it Let $\sigma > 0$ be a constant. If $h_K/\rho_K \le \sigma$ and
$h_K \le 1$, then there exists a constant $C = C(\sigma)$ independent
of $h_K$ such that}
\[
   \|v - \Ih v\|_{1,2,K} \le C h_K |v|_{2,2,K}, \qquad
   \forall v \in H^2(K).
\]
\item \textbf{The maximum angle condition},
 Babu\v{s}ka-Aziz \cite{BA}, Jamet \cite{J} (1976). \\
{\it Let $\theta_1$, $2\pi/3 \le \theta_1 < \pi$ be a constant.  If
any angle $\theta$ of $K$ satisfies $\theta \le \theta_1$ and
$h_K \le 1$, then there exists a constant $C = C(\theta_1)$ independent
of $h_K$ such that}
\[
   \|v - \Ih v\|_{1,2,K} \le C h_K |v|_{2,2,K}, \qquad
   \forall v \in H^2(K).
\]
\end{itemize}
It is easy to show that the minimum angle condition is equivalent
to the regularity condition \cite[Exercise~3.1.3, p130]{C}.
 Liu and Kikuchi presented an explicit form of the constant $C$
in \cite{LK}.

\edit{Inspired} by Liu-Kikuchi's result, Kobayashi
obtain\edit{ed} the following epoch-making result \cite{K1}, \cite{K2}.
Let $A$, $B$ \edit{and} $C$ be the lengths of the three edges of $K$ and
$S$ be the area of $K$.
\begin{itemize}
\item \textbf{Kobayashi's formula}, Kobayashi \cite{K1}, \cite{K2} \\
{\it Define the constant $C(K)$ by}
\[
   C(K) := \sqrt{\frac{A^2B^2C^2}{16S^2} - 
                    \frac{A^2 + B^2 + C^2}{30} - \frac{S^2}5
       \left(\frac1{A^2} + \frac1{B^2} + \frac1{C^2}\right)},
\]
{\it then the following holds}:
\[
   |v - \Ih v|_{1,2,K} \le C(K) |v|_{2,2,K}, \qquad
   \forall v \in H^2(K).
\]
\end{itemize}
Let $R_K$ be the radius of the circumcircle of $K$.  \edit{Using} the
formula $R_K = ABC/4S$, we \edit{can show that} $C(K) < R_K$ and obtain
a corollary of Kobayashi's formula.
\begin{itemize}
\item \textbf{A corollary of Kobayashi's formula} \\
{\it For any triangle $K \subset \R^2$, the following estimate holds}:
\[
   |v - \Ih v|_{1,2,K} \le R_K |v|_{2,2,K}, \qquad
   \forall v \in H^2(K).
\]
\end{itemize}
This corollary demonstrates that even if the
minimum angle is very small or the maximum angle is very close to $\pi$,
the error $|v - \Ih v|_{1,K}$ converges to $0$ if $R_K$ converges to $0$.
For example, consider \edit{the} isosceles triangle $K$ depicted in
Figure~1.  If $ 0 < h < 1$ and $\alpha > 1$, then $h^\alpha < h$ and the
circumradius of $K$ is $h^\alpha/2 + h^{2-\alpha}/8$.  Hence, if
$1 < \alpha < 2$ and $|v|_{2,2,K}$ is bounded, the error
$|v - \Ih v|_{1,2,K}$ converges to $0$ \edit{even though} the maximum
angle is \edit{tending} to $\pi$ as
$h\to 0$, although \edit{the} convergence rate becomes inferior. 

\begin{figure}[thb]
\begin{center}
  \psfrag{h}[][]{$h$}
  \psfrag{k}[][]{$h^\alpha$}
  \includegraphics[width=5cm]{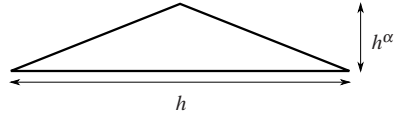} \\[0.0cm]
\caption{An example of \edit{a triangle} which violates the maximum
angle condition but \edit{satisfies} $R_K \to 0$ as $h\to 0$.}
\end{center}
\end{figure}

Suppose that $\{\tau_h\}_{h>0}$ is a series of triangulation\edit{s}
of a convex polygonal domain $\Omega\subset\R^2$ such that
\begin{equation}
  \lim_{h\to 0} \max_{K \in \tau_h} R_K =0.
  \label{circumradius-condition}
\end{equation}
Let $S_{\tau_h}$ be the set of all piecewise linear functions
on $\tau_h$\edit{,} defined by
\[
   S_{\tau_h} := \left\{v_h \in H_0^1(\Omega) \cap C(\overline{K}) \bigm|
    v_h|_{K} \in \mathcal{P}_1, \forall K \in \tau_h  \right\}.
\]
Let $u_h$ be the piecewise linear finite element solution 
on the triangulation $\tau_h$ of the Poisson problem
\[
    - \Delta u = f \text{ in } \Omega, \qquad
  u = 0 \text{ on }  \partial\Omega
\]
for a given $f \in L^2(\Omega)$.  Then, the
well-known C\'ea's lemma \cite[Theorem~2.4.1]{C} claims that,
for the exact solution $u$,
\begin{align*}
    \|u - u_h\|_{1,2,\Omega} & \le 
   \edit{\left(1 + C_P^2\right)^{1/2}|u - u_h|_{1,2,\Omega} \le}
  \left(1 + C_P^2\right)^{1/2}
   \inf_{v_h \in S_{\tau_h}}
   \edit{|} u - v_h \edit{|}_{1,2,\Omega} \\
  & \le \left(1 + C_P^2\right)^{1/2}
   \edit{|} u - \Ih u\edit{|}_{1,2,\Omega}
  \le \left(1 + C_P^2\right) \edit{^{1/2}}
   \left(\max_{K \in \tau_h} R_K\right) |u|_{2,2,\Omega},
\end{align*}
where $C_P$ is the Poincar\'e constant for $\Omega$.\footnote{
Poincar\'e's inequality claims that there exists a constant
$C_P > 0$ which is called the Poincar\'e constant
such that $|v|_{0,2,\Omega} \le C_{P}|v|_{1,2,\Omega}$ for
$v\in H_0^1(\Omega)$. See \cite{KJF}.}
 Thus, \textit{the discretization error} $\|u - u_h\|_{1,2,\Omega}$ is
bounded by \textit{the interpolation error} $\|u - \Ih u\|_{1,2,\Omega}$,
and the finite element solutions $\{u_h\}$ converge to $u$ even if the
maximum angle condition is violated (see the example of triangulation in
Fig.~2).  Therefore, to obtain the error estimate of $\Ih v$ or to
\edit{ensure that the} finite element solutions converge to \edit{the}
exact solution, $\max_{K\in \tau_h}R_k$ is more
important than the minimum \edit{or} maximum angle\edit{s}.

\begin{figure}[thb]
\begin{center}
  \includegraphics[width=9cm]{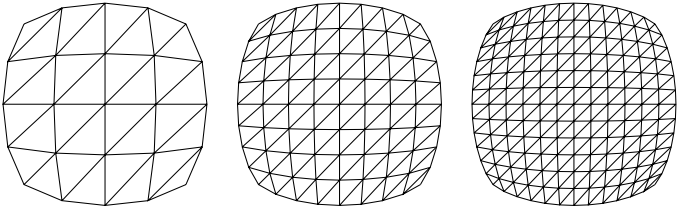} \\[0.5cm]
 \caption{A series of triangulation\edit{s} of the domain
   $|x-y|^{3/2}+|x+y|^{3/2} < 2$\edit{,} which satisf\edit{y} 
the circumradius condition.
The maximum angle condition is clearly violated.}
\end{center}
\end{figure}

A drawback of Kobayashi's formula is that its proof is very long
and needs \edit{the} assistance of validated numerical computation\edit{s}.
\edit{However,} in many cases the following estimation is good enough for
\edit{the} error analysis of finite element methods:
\begin{itemize}
\item \textbf{The circumradius condition}  \\
\textit{For an arbitrary triangle $K$ with $R_K \le 1$, there
exists a constant $C_p$ independent of $K$
such that the following estimate holds}: 
\[
   \|v - \Ih v\|_{1,p,K} \le C_p R_K |v|_{2,p,K}, \qquad
   \forall v \in W^{2,p}(K), \quad 1 \le p \le \infty.
\]
\end{itemize}
This estimation and/or the condition \Ref{circumradius-condition} are
called \textbf{the circumradius condition}.

The purpose of this paper is to \edit{prove} the circumradius
condition \edit{without using} validated numerical computation\edit{s}.
The main tool of our proof is the orthogonal expansion-contraction
transformation $F_{\alpha,\beta}:K \to \R^2$ defined by
$\Fab(x,y) := (\alpha x, \beta y)$ for $(x,y) \in K$ with
$\alpha, \beta > 0$.  Applying the orthogonal expansion-contraction
transformations twice, any triangle $K$ with circumradius $R$ becomes similar  
to the reference triangle $\hK$ whose apexes are $(0,0)$, $(1,0)$
and $(0,1)$.  Then, \edit{we estimate} the ratio of $|v|_{2,p,K}$ and
$|v|_{1,p,K}$
for $v \in W^{2,p}(K)$ \edit{using} a technique given by
Babu\v{s}ka-Aziz.  See the proof of Lemma~\ref{L3}.
In this sense, this paper is an extension of \cite{BA}.

\section{Preliminary and basic lemmas}
Let $K \subset \R^2$ be any triangle.
Partial derivatives of a function $u$ with respect to 
$x$, $y$ are denoted by
\[
   u_{x} := \frac{\partial u}{\partial x}, \qquad
   u_{xx} := \frac{\partial^2 u}{\partial x^2}, \quad
   u_{xy} := \frac{\partial^2 u}{\partial x\partial y}, \quad
  \mathrm{etc.} \qquad  (x,y) \in K.
\]
The usual Lebesgue and  Sobolev spaces on $K$
are denoted by $L^p(K)$, $W^{k,p}(K)$, $k=1,2$,
$p \in [1,\infty]$. As usual, $W^{k,2}(K)$ are denoted by
$H^k(K)$.  We denote their norms and semi-norms by
$|u|_{0,p,K}^p := \int_K |u|^p \dd \bfx$ for $p \in [1,\infty)$,
$|u|_{0,\infty,K} := \mathrm{ess} \sup_{K}|u|$ and
\begin{align*}
  &  |u|_{1,p,K}^p := |u_x|_{0,p,K}^p + |u_y|_{0,p,K}^p, \quad
    \|u\|_{1,p,K}^p := |u|_{0,p,K}^p + |u|_{1,p,K}^p, \\
  &  |u|_{2,p,K}^p := |u_{xx}|_{0,p,K}^p + |u_{yy}|_{0,p,K}^p
         + 2 |u_{xy}|_{0,p,K}^p, \\
  & |u|_{1,\infty,K} := \max\{
           |u_x|_{0,\infty,K}, |u_y|_{0,\infty,K}\},
    \|u\|_{1,\infty,K} := \max\{|u|_{0,\infty,K}, |u|_{1,\infty,K}\}\edit{,} \\
  & |u|_{2,\infty,K} := \max\{
           |u_{xx}|_{0,\infty,K}, |u_{yy}|_{0,\infty,K},
           |u_{xy}|_{0,\infty,K}\}.
\end{align*}

Throughout this paper and \cite{BA}, the most important tool is
the \textbf{orthogonal expansion-contraction} 
(OEC) transformation $F_{\alpha,\beta}:K \to \R^2$ defined,
for $\alpha, \beta > 0$, \edit{as}
\[
   \Fab(x,y) := (\alpha x, \beta y), \qquad
   (x,y) \in K.
\]
Define $\Kab:= \Fab(K)$ and take \edit{the}
arbitrary function $v \in W^{2,p}(\Kab)$.  Then, defining
$u := v \circ \Fab \in W^{2,p}(K)$, we have, for $p \in [1,\infty)$,
\begin{gather}
  \frac{|v|_{1,p,\Kab}^p}{|v|_{0,p,\Kab}^p}
    = \frac{\beta^p |u_{x}|_{0,p,K}^p
                + \alpha^p |u_{y}|_{0,p,K}^p}
        {\alpha^p\beta^p |u|_{0,p,K}^p},  \label{trans1} \\
  \frac{|v|_{2,p,\Kab}^p}{|v|_{0,p,\Kab}^p}
    = \frac{\frac{\beta^p}{\alpha^p} |u_{xx}|_{0,p,K}^p
                + \frac{\alpha^p}{\beta^p} |u_{yy}|_{0,p,K}^p
                  + 2 |u_{xy}|_{0,p,K}^p}
          {\alpha^p\beta^p |u|_{0,p,K}^p},  \label{trans2} \\
  \frac{|v|_{2,p,\Kab}^p}{|v|_{1,p,\Kab}^p}
    = \frac{\frac{\beta^p}{\alpha^p} |u_{xx}|_{0,p,K}^p
                + \frac{\alpha^p}{\beta^p} |u_{yy}|_{0,p,K}^p
                  + 2 |u_{xy}|_{0,p,K}^p}
        {\beta^p |u_x|_{0,p,K}^p + \alpha^p |u_y|_{0,p,K}^p}
     \label{trans3}
\end{gather}
and
\begin{gather}
  \frac{|v|_{1,\infty,\Kab}}{|v|_{0,\infty,\Kab}}
    = \frac{\max \left\{ \beta|u_{x}|_{0,\infty,K},
                 \alpha |u_{y}|_{0,\infty,K} \right\}}
        {\alpha\beta |u|_{0,\infty,K}},  \notag \\
  \frac{|v|_{2,\infty,\Kab}}{|v|_{0,\infty,\Kab}}
    = \frac{\max \left\{\frac{\beta}{\alpha} |u_{xx}|_{0,\infty,K},
                \frac{\alpha}{\beta} |u_{yy}|_{0,\infty,K},
                 |u_{xy}|_{0,\infty ,K} \right\} }
          {\alpha\beta |u|_{0,\infty,K}},  \notag \\
  \frac{|v|_{2,\infty,\Kab}}{|v|_{1,\infty,\Kab}}
    = \frac{ \max \left\{  \frac{\beta}{\alpha} |u_{xx}|_{0,\infty,K},
                \frac{\alpha}{\beta} |u_{yy}|_{0,\infty,K},
                |u_{xy}|_{0,\infty,K}  \right\}}
        {\max \left\{\beta |u_x|_{0,\infty,K},
           \alpha |u_y|_{0,\infty,K} \right\} }.
     \label{trans3-inf}
\end{gather}

Let $\hK$ be the reference triangle whose apexes are
$(0,0)$, $(1,0)$, $(0,1)$.  Take $\alpha$, $\beta$ so that
\begin{equation}
 \alpha^2 + \beta^2 = 2, \qquad
  0 < \beta \le 1 \le \alpha < \sqrt{2},
\end{equation}
and define $\Kab := \Fab(\hK)$.  Note that $\hK = K_{1,1}$
and the circumradii of $\hK$ and $\Kab$ are $1/\sqrt{2}$.

For $p \in [1,\infty]$, 
define $\Xi_{p}^1(\Kab)$, $\Xi_{p}^2(\Kab)$, $\T_p(\Kab)$ by
\begin{gather*}
   \Xi_{p}^1(\Kab) := 
  \left\{v \in W^{1,p}(\Kab) \Bigm|
      \int_0^\alpha v(x,0) \dd x = 0 \right\}, \\
   \Xi_{p}^2(\Kab) := 
  \left\{v \in W^{1,p}(\Kab) \Bigm|
        \int_0^\beta v(0,y) \dd y = 0 \right\}, \\
  \T_p(\Kab) := 
    \left\{v \in W^{2,p}(\Kab) \Bigm|
     v(0,0) = v(\alpha,0) = v(0,\beta) = 0 \right\}.
\end{gather*}

The following lemma is an extension of 
\cite[Lemma~2.1]{BA} to any $p \in [1,\infty]$.
Although the proof is very similar, we \edit{include} it here 
for \edit{the} reader\edit{s'} convenience.

\begin{lemma} \label{L1} For $p \in [1,\infty]$,
define the constants $A_{p1}$, $A_{p2}$ by
\[
  A_{p1} := \inf_{w \in \Xi_{p}^1(\hK)} 
  \frac{|w|_{1,p,\hK}}{|w|_{0,p,\hK}}, \quad
  A_{p2} := \inf_{z \in \Xi_{p}^2(\hK)} 
  \frac{|z|_{1,p,\hK}}{|z|_{0,p,\hK}}.
\]
Then, we have $A_p := A_{p1} = A_{p2} > 0$.
\end{lemma}
\begin{proof}
The equality $A_{p1}=A_{p2}$ is clear from the symmetry of $\hK$.
The proof of $A_{p1} > 0$ is by contradiction. Suppose that $A_{p1} = 0$.
Then\edit{,} there exists $\{w_k\}_{k=1}^\infty \subset \Xi_{p}^1$ such that
\[
   |w_k|_{0,p,\hK} = 1, \qquad
   \lim_{k \to \infty} |w_k|_{1,p,\hK} = 0.
\]
Let $\mathcal{P}_0 = \R$ be the set of polynomials of degree $0$.
By \cite[Theorem~3.1.1]{C}, there is a constant $C(\hK,p)$ such that
\begin{align*}
  \inf_{q \in \mathcal{P}_0} \|v + q\|_{1,p,\hK} & \le 
  C(\hK,p) |v|_{1,p,\hK}, \qquad \forall v \in W^{1,p}(\hK).
\end{align*}
Therefore, there exists
$\{q_k\}\subset \mathcal{P}_0$ such that
\begin{gather*}
   \inf_{q \in \mathcal{P}_0}
  \|w_k + q\|_{1,p,\hK} 
   \le \|w_k + q_k\|_{1,p,\hK}
   \le \inf_{q \in \mathcal{P}_0}
  \|w_k + q\|_{1,p,\hK} + \frac{1}{k}, \\
  \lim_{k \to \infty} \|w_k + q_k\|_{1,p,\hK} \le
  \lim_{k \to \infty}
  \left(C(\hK,p)|w_k|_{1,p,\hK} + \frac{1}k \right)= 0. 
\end{gather*}
Since the sequence $\{w_k\}\subset W^{1,p}(\hK)$ is bounded,
$\{q_k\} \subset \mathcal{P}_0 = \R$ is \edit{also} bound\edit{ed}.
Thus, there exists a subsequence $\{q_{k_i}\}$ \edit{such} that
$q_{k_i}$ converges to $\bar{q} \in \mathcal{P}_0$.
In particular, we have
\[
   \lim_{k_i \to \infty} \|w_{k_i} + \bar{q}\|_{1,p,\hK} = 0.
\]
Let $\gamma : W^{1,p}(\hK) \to W^{1-1/p,p}(\Gamma)$ be the trace
operator, where $\Gamma$ is the edge of $\hK$ connecting $(1,0)$ and
$(0,0)$.  The boundedness of $\gamma$ and the inclusion
$W^{1-1/p,p}(\Gamma) \subset L^1(\Gamma)$ yield
\[
 0 = \lim_{k \to \infty} \int_\Gamma \gamma(w_{k_i} + \bar{q}) \dd s
  = \int_\Gamma \bar{q} \dd s,
\]
since $w_{k_i} \in \Xi_{p}^1$.
Hence, we conclude that $\bar{q} = 0$ and 
$\lim_{k_i \to \infty}\|w_{k_i}\|_{1,p,\hK} = 0$, which
contradicts $\lim_{k_i \to \infty}|w_{k_i}|_{0,p,\hK}=1$. 
\footnote{For the trace operator, see, for example, \cite{KJF}.
If $p = \infty$, the boundedness of $\gamma$ is obvious since
$W^{1,\infty}(\Omega) = C^{0,1}(\overline{\Omega})$
and $W^{1,\infty}(\Gamma) = C^{0,1}(\overline{\Gamma})$.}  \qed
\end{proof}

\begin{remark}
The constant $1/A_2$ is called the \textbf{Babu\v{s}ka-Aziz
constant}.  According to \cite[pp40-41]{LK}, the Babu\v{s}ka-Aziz
constant $1/A_2$ is the maximum positive solution of the equation
$1/x + \tan(1/x) = 0$ and its approximated value is
$1/A_2 \approx 0.49291$.
\end{remark}

\begin{lemma}\label{L2}
Define the constants $A_{p1}(\Kab)$, $A_{p2}(\Kab)$ by
\[
  A_{p1}(\Kab) := \inf_{w \in \Xi_{p}^1(\Kab)} 
  \frac{|w|_{1,p,\Kab}}{|w|_{0,p,\Kab}}, \;
  A_{p2}(\Kab) := \inf_{z \in \Xi_{p}^2(\Kab)} 
  \frac{|z|_{1,p,\Kab}}{|z|_{0,p,\Kab}}.
\]
\edit{Then} $A_{p1}(\Kab) \ge A_p/\sqrt{2}$,
$A_{p2}(\Kab) \ge  A_p/\sqrt{2}$.
\end{lemma}
\begin{proof}  Suppose that $p \in [1,\infty)$. Then, 
\Ref{trans1} yields, for any $v \in \Xi_{p}^i(\Kab)$, $i=1,2$,
\begin{align*}
  \frac{|v|_{1,p,\Kab}^p}{|v|_{0,p,\Kab}^p}
    = \frac{|u_{x}|_{0,p,\hK}^p 
                + \frac{\alpha^p}{\beta^p} |u_{y}|_{0,p,\hK}^p}
        {\alpha^p |u|_{0,p,\hK}^p}
   \ge \frac{|u_{x}|_{0,p,\hK}^p + |u_{y}|_{0,p,\hK}^p}
        {2^{p/2} |u|_{0,p,\hK}^p} \ge \frac{A_p^p}{2^{p/2}}.
\end{align*}
Taking \edit{the} infimum of the left-hand side with respect to
$v \in \Xi_p^i(\Kab)$, we obtain $A_{pi}(\Kab) \ge A_p/\sqrt{2}$.
The case $p = \infty$ \edit{is similarly proved}. \qed
\end{proof}

The proof of the following lemma is very similar to that of
Babu\v{s}ka-Aziz's \cite[Lemma~2.2]{BA}.  We present it here
for \edit{the} reader\edit{s'} convenience.
\begin{lemma} \label{L3}
Define the constant $B_p(\Kab)$ by
\[
  B_p(\Kab) := \inf_{v \in \T_p(\Kab)} 
  \frac{|v|_{2,p,\Kab}}{|v|_{1,p,\Kab}}.
\]
\edit{Then} $B_p(\Kab) \ge A_p/\sqrt{2}$.
\end{lemma}
\begin{proof} Let $1 \le p < \infty$.  Take any $v \in \T_p(\Kab)$
and define $u:=v\circ\Fab \in \T_p(\hK)$.  It follows from \Ref{trans3}
and $\alpha/\beta \ge 1$ that
\begin{align*}
  \frac{|v|_{2,p,\Kab}^p}{|v|_{1,p,\Kab}^p}
 &   = \frac{|u_{xx}|_{0,p,\hK}^p
      + \frac{\alpha^p}{\beta^p} |u_{xy}|_{0,p,\hK}^p
      + \frac{\alpha^p}{\beta^p} 
    \left[|u_{xy}|_{0,p,\hK}^p +
   \frac{\alpha^p}{\beta^p} |u_{yy}|_{0,p,\hK}^p\right]}
        {\alpha^p \left(|u_{x}|_{0,p,\hK}^p
       + \frac{\alpha^p}{\beta^p} |u_{y}|_{0,p,\hK}^p\right)} \\
 &   \ge \frac{|u_{xx}|_{0,p,\hK}^p + |u_{xy}|_{0,p,\hK}^p +
   \frac{\alpha^p}{\beta^p} 
    \left[|u_{xy}|_{0,p,\hK}^p + |u_{yy}|_{0,p,\hK}^p\right]}
        {2^{p/2} \left(|u_{x}|_{0,p,\hK}^p 
    + \frac{\alpha^p}{\beta^p} |u_{y}|_{0,p,\hK}^p\right)}.
\end{align*}
Setting $w := u_{x}$, we notice $w \in \Xi_{p}^1(\hK)$, and
\begin{align*}
  |u_{xx}|_{0,p,\hK}^p + |u_{xy}|_{0,p,\hK}^p
   = |w|_{1,p,\hK}^p  \ge A_p^p |w|_{0,p,\hK}^p
   = A_p^p |u_{x}|_{0,p,\hK}^p
\end{align*}
by Lemma~\ref{L3}.  Similarly, setting
$z := u_{y}$, we have $z \in \Xi_{p}^2(\hK)$ and hence
\begin{align*}
  |u_{xy}|_{0,p,\hK}^p + |u_{yy}|_{0,p,\hK}^p
   = |z|_{1,p,\hK}^p  \ge A_p^p |z|_{0,p,\hK}^p
   = A_p^p |u_{y}|_{0,p,\hK}^p.
\end{align*}
Therefore,
\[
  \frac{|v|_{2,p,\Kab}^p}{|v|_{1,p,\Kab}^p}
  \ge   \frac{A_p^p
      \left(|u_x|_{0,p,\hK}^p
           + \frac{\alpha^p}{\beta^p} |u_y|_{0,p,\hK}^p\right)}
        {2^{p/2}\left(|u_x|_{0,p,\hK}^p
           + \frac{\alpha^p}{\beta^p} |u_y|_{0,p,\hK}^p\right)}
    = \frac{A_p^p}{2^{p/2}}.
\]
Taking \edit{the} infimum with respect to $v$, we conclude 
$B_p(\Kab) \ge {A_p}/{\sqrt{2}}$. 
The proof of the case $p=\infty$ is similar.  \qed
\end{proof}

The following lemma is an extension of 
\cite[Lemma~2.3]{BA} to any $p \in [1,\infty]$.
\edit{As the proof is relatively simple}, we omit \edit{the details}.

\begin{lemma} \label{L4} Define the constant $D_p$ by
\[
  D_p := \inf_{u \in \T_p(\hK)} 
  \frac{|u|_{2,p,\hK}}{|u|_{0,p,\hK}}.
\]
\edit{Then} $D_p > 0$.
\end{lemma}

\vspace{0.3cm}
\begin{lemma} \label{L5} Define the constant $D_p(\Kab)$ by
\[
  D_p(\Kab) := \inf_{v \in \T_p(\Kab)} 
  \frac{|v|_{2,p,\Kab}}{|v|_{0,p,\Kab}}.
\]
\edit{Then} $D_p(\Kab) \ge D_p/2$.
\end{lemma}
\begin{proof}
Suppose that $p \in [1,\infty)$. Then, 
\Ref{trans2} yields, for any $v \in \T_p(\Kab)$,
\begin{align*}
  \frac{|v|_{2,p,\Kab}^p}{|v|_{0,p,\Kab}^p}
    = \frac{|u_{xx}|_{0,p,\hK}^p 
                + 2 \frac{\alpha^p}{\beta^p} |u_{xy}|_{0,p,\hK}^p
                + \frac{\alpha^{2p}}{\beta^{2p}} |u_{yy}|_{0,p,\hK}^p}
        {\alpha^{2p} |u|_{0,p,\hK}^p} \\
     \ge \frac{|u_{xx}|_{0,p,\hK}^p + |u_{yy}|_{0,p,\hK}^p
                    + 2 |u_{xy}|_{0,p,\hK}^p}
        {2^{p} |u|_{0,p,\hK}^p} \ge \frac{D_p(\hK)^p}{2^{p}}.
\end{align*}
Taking \edit{the} infimum of the left-hand side with respect to
$v \in \T_p(\Kab)$, we obtain $D_p(\Kab) \ge D_p(K)/2$.
The case $p = \infty$ can be \edit{proved} in the same manner.  \qed
\end{proof}

\begin{remark}
According to \cite[pp40-41]{LK}, the approximated value of $D_2$ is
$1/0.167$.
\end{remark}

\section{The circumradius condition for right triangular elements}
\label{sect3}
Take  $R > 0$ and let the linear map $G_R$ be defined by
\[
   G_R:\R^2 \to \R^2, \qquad 
   G_R(\bfx) := R \bfx, \qquad \bfx \in \R^2.
\]
Two bounded domains $\Omega_1$, $\Omega_2 \subset \R^2$ are
called \textbf{similar} if there exists a map $\varphi$ which
consists of a rotation and a parallel translation,
such that $\Omega_2 = \varphi\circ G_R(\Omega_1)$.
If $\varphi$ is a parallel translation, $\varphi$ preserves \edit{the} Sobolev
norms of functions in $W^{m,p}(\Omega_1)$, ($m=0,1,2$,
$p \in [1,\infty]$).  Hence, we may ignore $\varphi$ \edit{in the following
discussion without loss of generality}.
Set $\Kab^{R} := G_{\sqrt{2}R}(\Kab)$.  The circumradius of  $\Kab^R$ is
$R$.  Take $v \in W^{m,p}(\Kab^R)$ and define
$\bar{v}:=v\circ G_{\sqrt{2}R} \in W^{m,p}(\Kab)$.  Then, we have
\begin{equation}
   |v|_{m,p,\Kab^R} = (\sqrt{2}R)^{2/p-m}|\bar{v}|_{m,p,\Kab}, \qquad
   m = 0,1,2.
  \label{similar}
\end{equation}
For the domain $\Kab^R$, we define $\T_p(\Kab^R)$ as before and 
\edit{find that}
\begin{align*}
  B_p(\Kab^R) := & \inf_{v \in \T_p(\Kab^R)} 
  \frac{|v|_{2,p,\Kab^R}}{|v|_{1,p,\Kab^R}}
  = \frac{B_p(\Kab)}{\sqrt{2}R}
  \ge \frac{A_p}{2R}, \\
  D_p(\Kab^R) := & \inf_{v \in \T_p(\Kab^R)} 
  \frac{|v|_{2,p,\Kab^R}}{|v|_{0,p,\Kab^R}}
  = \frac{D_p(\Kab)}{2R^2}
  \ge \frac{D_p}{4R^2}.
\end{align*}
Combining these estimates we obtain the following lemma.

\begin{lemma}\label{L6}
Let $K \subset \R^2$ be a right triangle whose circumradius is $R$.
Suppose that the two edges which contain the right angle are parallel
to x- and y-axis.  Then, we have the following estimates: 
\begin{gather}
  B_p(K) \ge \frac{A_p}{2R}, \qquad
  D_p(K)  \ge \frac{D_p}{4R^2}, \qquad 1 \le p \le \infty.
  \label{lemma6}
\end{gather}
\end{lemma}

As is stated in the introduction, the linear interpolation operator
$\Ih v \in \mathcal{P}_1$ for $v \in W^{2,p}(K_R)$ is defined by
$(\Ih v)(\bfx_i) = v(\bfx_i)$, where ${\bf x}_i$, $i = 1,2,3$
are apexes of $K_R$.

\begin{theorem}\label{T1}
Let $K\subset\R^2$ be a right triangle whose circumradius is $R$.
Suppose that the two edges \edit{that} contain the right angle are parallel
to \edit{the} x- and y-ax\edit{e}s.  Then, the error of $\Ih$ on $K$ is estimated as
\begin{align*}
  \|v - \Ih v\|_{1,p,K} & \le C_p R |v|_{2,p,K}, \quad
   \forall v\in W^{2,p}(K), \quad 1 \le p \le \infty, \\
   C_p & := \begin{cases}
   \left(2^{p}A_p^{-p} + 4^pR^{p}
 D_p^{-p}\right)^{1/p}  & 1 \le p < \infty, \\
  \max\{2/A_\infty, 4R/D_\infty\} & p = \infty.
  \end{cases}
\end{align*}
\end{theorem}
\begin{proof}
Since $v - Pv \in \T_p(K)$ for $v \in W^{2,p}(K)$, we may apply
Lemma~\ref{L6} and obtain
\[
   \|v - \Ih v\|_{1,p,K} \le C_p R |v - \Ih v|_{2,p,K}
    = C_p R |v|_{2,p,K}.
\]
\qed
\end{proof}

\section{The circumradius condition for general triangular elements}
In this section, we prove the circumradius condition for
general triangular elements. Let $\Kst$ be the right triangle
with apexes $N_1(-1,0)$, $N_2(1,0)$, $N_3(s,t)$,
 $t > 0$, $s^2 + t^2 = 1$.  The circumradius of $\Kst$ is $1$.
Define $\KK := F_{1,\eta}(\Kst)$ using the OEC transformation
$F_{1,\eta}$.  Note that any triangle $K$ is similar
to $\KK$ with appropriate $(s,t)$ and $\eta > 0$ (see Fig.~3).
We then try to write a lower bound of
\[
   \inf_{v \in \T_p(\KK)}
  \frac{|v|_{2,p,\KK}}{|v|_{1,p,\KK}}
\]
using $A_p$ and the circumradius of $\KK$.
We may assume without loss of generality that
the baseline $N_1N_2$ is the longest edge of $\KK$.
Under this assumption, $\eta$ is in the interval $(0, \sqrt{3}]$.

\begin{figure}[thb]
\begin{center}
  \includegraphics[width=10cm]{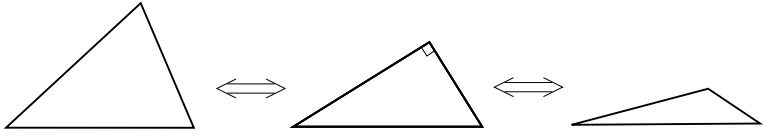} \\[0.5cm]
 \caption{$\Kst$ and $\KK$ with $\eta \in (0,\sqrt{3}]$.}
\end{center}
\end{figure}

\subsection{Preliminary}
Define the constants $a$, $b$, $X$, $Y$ by 
\[
  (a,b) := \frac1{\sqrt{2}}\left(\sqrt{1 + s}, \sqrt{1 - s}\right),
  \qquad  X := \sqrt{a^2 \eta^2 + b^2}, \quad
          Y := \sqrt{a^2 + b^2 \eta^2}.
\]
Note that $a^2 + b^2 = 1$, $2ab = t$ and
the vector $(a,b)$ is parallel to the edge $N_1N_3$.
We also have
\begin{align}
   X \le \sqrt{3a^2 + 3b^2} \le \sqrt{3}, \quad
   Y \le \sqrt{3}, \label{eta1} \\
   \frac{X}{\eta} = \sqrt{a^2 + \frac{b^2}{\eta^2}} \ge
 \frac{1}{\sqrt{3}}, \quad \frac{Y}{\eta} \ge \frac{1}{\sqrt{3}}.
     \label{eta2}
\end{align}
Note that the circumradius $R$ of $\KK$ is
\begin{align*}
 R = \frac{2\sqrt{(1-s)^2 + \eta^2t^2}\sqrt{(1+s)^2 + \eta^2t^2}}{4\eta t}
   = \frac{\sqrt{a^2 + b^2\eta^2}\sqrt{b^2 + a^2\eta^2}}{\eta}
   = \frac{XY}{\eta}.
\end{align*}
We also observe that the inequality
\begin{align*}
   X^2Y^2  = \eta^2 + a^2 b^2 (1 - \eta^2)^2 \ge a^2 b^2 (1 - \eta^2)^2
\end{align*}
yields
\begin{gather}
 XY - ab|1-\eta^2| = \frac{X^2Y^2-a^2b^2(1-\eta^2)^2}{XY+ab|1-\eta^2|}
  \ge \frac{\eta^2}{2XY}, \notag \\
  XY - \frac{\eta^2}{2XY} \ge ab|1-\eta^2|.
  \label{XY}
\end{gather}
We notice that the following inequalities hold
for any positive numbers $U$, $V$ and $p\ge1$:
\begin{gather}
    U^p + V^p \le 2^{\tau(p)} (U^2 + V^2)^{p/2}, \quad
   \tau(p) := \begin{cases}
         1-p/2, & 1 \le p \le 2 \\
         0,     & 2 \le p < \infty
     \end{cases},  \label{tau}\\
    (U^2 + V^2)^{p/2} \le 2^{\gamma(p)} (U^p + V^p), \quad
   \gamma(p) := \begin{cases}
         0, &  1 \le p \le 2 \\
         p/2 - 1,     & 2 \le p < \infty
     \end{cases}.
    \label{gamma}
\end{gather}

\subsection{A congruent transformation}
Define the congruent transformation $F:(x,y) \mapsto (z,w)$ by
\[
     \begin{pmatrix} z \\ w \end{pmatrix} =
  \begin{pmatrix} -a & -b \\ b & -a \end{pmatrix}
  \begin{pmatrix} x -s \\ y - t \end{pmatrix}.
\]
This transformation defines the $(z,w)$-coordinate on $\Kst$ and
maps the three apexes $(-1,0)$, $(1,0)$, $(s,t)$ of $\Kst$ 
to $(2a,0)$, $(0,2b)$, $(0,0)$. See Figure~4.
\begin{figure}[thb]
\begin{center}
  \psfrag{x}[][]{\small $x$}
  \psfrag{y}[][]{\small $y$}
  \psfrag{z}[][]{\small $z$}
  \psfrag{w}[][]{\small $w$}
  \psfrag{a}[][]{\small $-1$}
  \psfrag{b}[][]{\small $1$}
  \psfrag{c}[][]{\small $(s,t)$}
  \includegraphics[width=6.5cm]{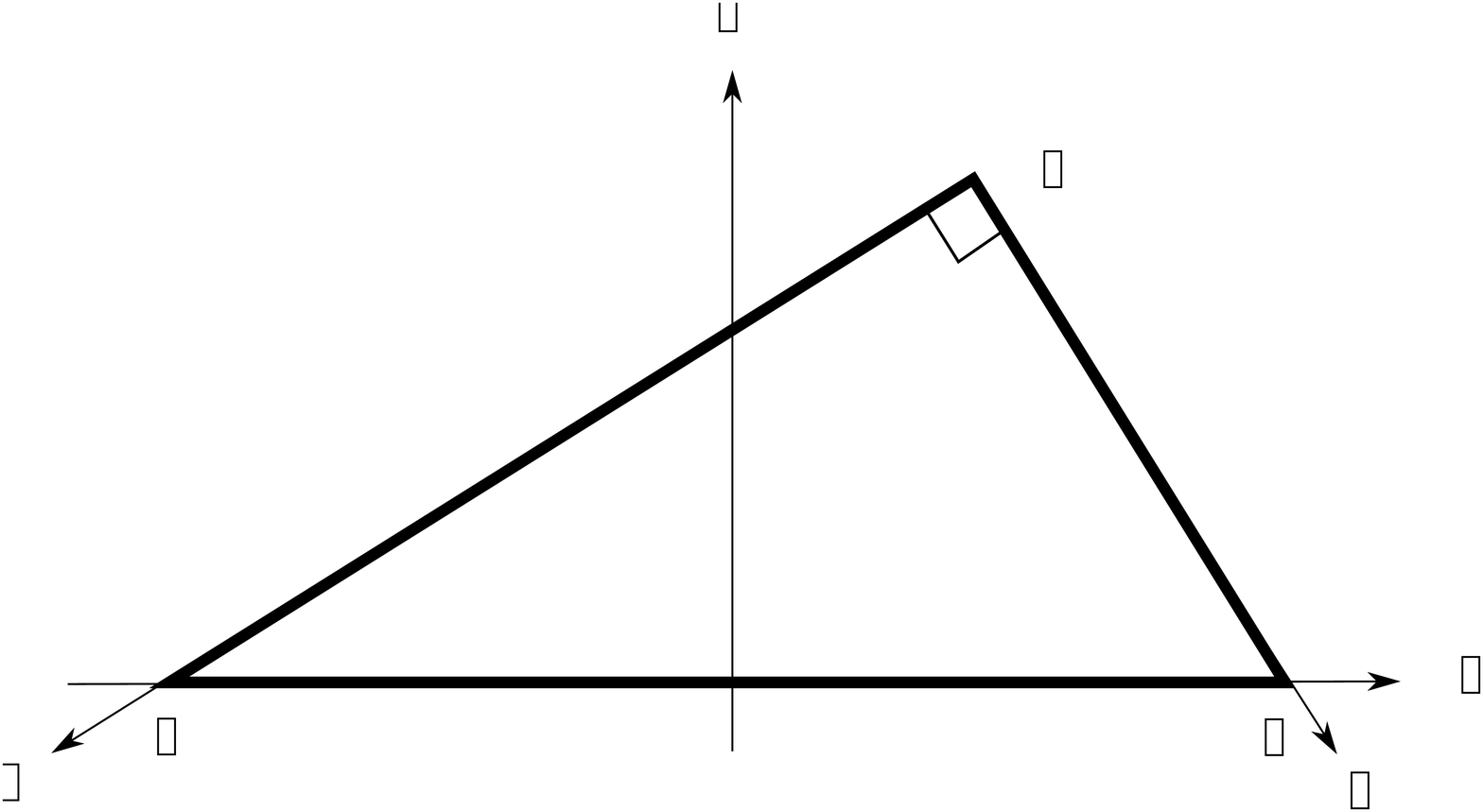} \\[0.5cm]
 Fig.~4: \textit{The two coordinates on $\Kst$.}
\end{center}
\end{figure}

For a sufficiently smooth function $f$ defined on $\Kst$,
we \edit{can} write upper and lower bounds of
$\eta^p|f_x|_{0,p,\Kst}^p + |f_y|_{0,p,\Kst}^p$
using $|f_z|_{0,p,\Kst}$, $|f_w|_{0,p,\Kst}$.
Let us suppose $1 \le p < \infty$ at first.  It follows from
\Ref{XY}, \Ref{tau}, \Ref{gamma} and
\[
    f_x = - a f_z + b f_w, \qquad f_y = -b f_z - a f_w,
\]
that
\begin{align}
  \eta^p |f_x|_{0,p,\Kst}^p & + |f_y|_{0,p,\Kst}^p
  = \int_{\Kst} \left(\eta^p |-a f_z + b f_w|^p
   + |b f_z + a f_w|^p\right) \dd \bfx \notag \\
 (\text{by \Ref{tau}}) \quad
  & \le 2^{\tau(p)} \int_{\Kst} \left(\eta^2 (-a f_z + b f_w)^2
   + (b f_z + a f_w)^2\right)^{p/2} \dd \bfx \notag \\
  & = 2^{\tau(p)} \int_{\Kst} \bigl(X^2|f_z|^2
   + Y^2|f_w|^2 
   + 2ab(1-\eta^2)f_zf_w\bigr)^{p/2} \dd \bfx \notag \\
 (\text{by \Ref{XY}}) \quad
  & \le 2^{\tau(p)} \int_{\Kst} \left(X|f_z|
   + Y|f_w|\right)^{p} \dd \bfx \notag \\
  & \le 2^{\tau(p)+p-1} \left(X^p|f_z|_{0,p,\Kst}^p
   + Y^p|f_w|_{0,p,\Kst}^p\right)\edit{,} \label{Pi-D}
\end{align}
and
{\allowdisplaybreaks 
\begin{align}
  \eta^p |f_x|_{0,p,\Kst}^p & + |f_y|_{0,p,\Kst}^p \notag \\
%
   (\text{by \Ref{gamma}})\; 
 &  \ge  2^{-\gamma(p)} \int_{\Kst} \left(X^2|f_z|^2
   + Y^2|f_w|^2 + 2ab(1-\eta^2)f_zf_w\right)^{p/2} \dd \bfx \notag \\
  & \ge 2^{-\gamma(p)} \int_{\Kst} \biggl(X^2|f_z|^2
   + Y^2|f_w|^2  \notag \\
  & \hspace{2cm} -ab|1-\eta^2|
   \left(\frac{X}{Y}|f_z|^2 + \frac{Y}{X}|f_w|^2\right)\biggr)^{p/2}
   \dd \bfx \notag\\
(\text{by \Ref{XY}})\;
   & \ge 2^{-\gamma(p)} \int_{\Kst} \biggl(X^2|f_z|^2
   + Y^2|f_w|^2 \notag \\
  & \hspace{2cm}- \left(XY - \frac{\eta^2}{2XY}\right)
   \left(\frac{X}{Y}|f_z|^2 + \frac{Y}{X}|f_w|^2\right)\biggr)^{p/2}
   \dd \bfx \notag \\
  & = 2^{-\gamma(p)-p/2}\eta^p \int_{\Kst} \left(\frac{1}{Y^2}|f_z|^2
   + \frac{1}{X^2}|f_w|^2\right)^{p/2} \dd \bfx \notag \\
 (\text{by \Ref{tau}})\;
  & \ge 2^{-\gamma(p)-\tau(p)-p/2}\eta^p
   \left(\frac{|f_z|_{0,p,\Kst}^p }{Y^p}
   + \frac{|f_w|_{0,p,\Kst}^p}{X^p}\right). \label{Pi-N}
\end{align}
}
If $p=\infty$, we obtain 
\begin{gather*}
  \max\{\eta |f_x|_{0,\infty,\Kst}, |f_y|_{0,\infty,\Kst}\} \le
 2 \max\left\{X|f_z|_{0,\infty,\Kst}, Y|f_w|_{0,\infty,\Kst}\right\}\edit{,}
\end{gather*}
\edit{in} exactly \edit{the} same manner.
(In the \edit{following}, we \edit{have} denoted
the $L^p(\Kst)$-norm by $|\cdot|_{p}$.)
To obtain the opposite inequality,
let $\bfx \in \Kst$ be any point.
Then, using the previous \edit{technique} we have
\begin{align*}
  \eta^2 |f_x(\bfx)|^2 + |f_y(\bfx)|^2
  & = X^2 |f_x(\bfx)|^2 + Y^2 |f_y(\bfx)|^2
    + 2ab (1 - \eta^2) f_z(\bfx)f_w(\bfx) \\
  & \ge  X^2 |f_x(\bfx)|^2 + Y^2 |f_y(\bfx)|^2 \\
  & \qquad  - ab |1 - \eta^2|\left(
    \frac{X}{Y} |f_z(\bfx)|^2 + \frac{Y}{X}|f_w(\bfx)|^2 \right) \\
  & \ge \frac{\eta^2}2 \left(
   \frac{|f_z(\bfx)|^2}{Y^2} + \frac{|f_w(\bfx)|^2}{X^2}
   \right).
\end{align*}
This inequality yields
\begin{align*}
  \max\{\eta |f_x|_{\infty}, |f_y|_{\infty}\} & \ge
   \max\{\eta |f_x(\bfx)|, |f_y(\bfx)|\} \\
 & \ge \frac1{\sqrt{2}}
  \left(\eta^2 |f_x(\bfx)|^2 + |f_y(\bfx)|^2\right)^{1/2} \\
 & \ge \frac{\eta}2 \left(
   \frac{|f_z(\bfx)|^2}{Y^2} + \frac{|f_w(\bfx)|^2}{X^2}
   \right)^{1/2} \ge \frac{\eta|f_z(\bfx)|}{2Y}, \\
  \max\{\eta |f_x|_{\infty}, |f_y|_{\infty}\} & \ge
   \frac{\eta|f_w(\bfx)|}{2X}.
\end{align*}
Since $\bfx \in \Kst$ is arbitrary, we obtain
\[
 \max\{\eta |f_x|_{\infty}, |f_y|_{\infty}\}  \ge
  \frac{\eta}2 \max\left\{
   \frac{|f_z|_\infty}{Y}, \frac{|f_w|_\infty}{X} \right\}.
\]

\subsection{A proof of the circumradius condition}
We \edit{can now} prove the circumradius condition.
 Let $1 \le p < \infty$.
For $v \in \T_p(\KK)$, define
$u := v \circ F_{1,\eta} \in \T_p(\Kst)$. Then\edit{,} it follows from
\Ref{trans3} that
\begin{align*}
  \frac{|v|_{2,p,\KK}^p}{|v|_{1,p,\KK}^p}
   = \frac{\Pi_N^p}{\Pi_D^p}
      := \frac{{\eta^p} |u_{xx}|_{p}^p
             + \frac{1}{\eta^p} |u_{yy}|_{p}^p
             + 2 |u_{xy}|_{p}^p} {\eta^p |u_x|_{p}^p + |u_y|_{p}^p}.
\end{align*}
By \Ref{Pi-D} the denominator $\Pi_D^p$ is estimated as
\begin{align*}
  \Pi_D^p  \le 2^{\tau(p)+p-1} \left(X^p|u_z|_{p}^p
   + Y^p|u_w|_{p}^p\right).
\end{align*}
By \Ref{Pi-N} we notice
\begin{gather*}
   \eta^p |u_{xx}|_{p}^p  + |u_{xy}|_{p}^p
   \ge c_1 \eta^p
   \left(\frac{|u_{xz}|_{p}^p }{Y^p}
   + \frac{|u_{xw}|_{p}^p}{X^p}\right), \\
   \eta^p |u_{xy}|_{p}^p  + |u_{yy}|_{p}^p
   \ge c_1 \eta^p  \left(\frac{|u_{yz}|_{p}^p}{Y^p}
   + \frac{|u_{yw}|_{p}^p}{X^p}\right),
\end{gather*}
where $c_1:=  2^{-\gamma(p)-\tau(p)-p/2}$.
Hence, the numerator $\Pi_N^p$ is estimated as
{\allowdisplaybreaks
\begin{align*}
  \Pi_N^p & = \eta^p |u_{xx}|_{p}^p  + |u_{xy}|_{p}^p
   + \eta^{-p}\left(
     \eta^p |u_{xy}|_{p}^p  + |u_{yy}|_{p}^p\right) \\
 & \ge c_1 \bigg\{
  \eta^p\left(\frac{|u_{xz}|_{p}^p }{Y^p}
   + \frac{|u_{xw}|_{p}^p}{X^p}\right)
  + \frac{|u_{yz}|_{p}^p }{Y^p}
   + \frac{|u_{yw}|_{p}^p}{X^p}\bigg\} \\
 & = c_1 \biggl\{
  \frac{1}{Y^p}\left(\eta^p |u_{xz}|_{p}^p
    + |u_{yz}|_{p}^p \right) + 
   \frac{1}{X^p}\left(\eta^p |u_{xw}|_{p}^p
    + |u_{yw}|_{p}^p \right) \biggr\} \\
 & \ge c_1^2\eta^p \biggl(
   \frac{|u_{zz}|_{p}^p}{Y^{2p}} +
   \frac{2|u_{zw}|_{p}^p}{X^pY^p} +
   \frac{|u_{ww}|_{p}^p}{X^{2p}} \biggr) \\
  (\text{by \Ref{eta1}})\;
 & \ge c_1^2\, 3^{-p/2}\,\eta^p \biggl(
  \frac1{Y^p}\left(|u_{zz}|_{p}^p + |u_{zw}|_{p}^p\right)
 + \frac1{X^p}\left(|u_{zw}|_{p}^p 
    + |u_{ww}|_{p}^p\right)\biggr) \\
  & \ge c_1^2\, 2^{-p}\, 3^{-p/2}\,\eta^p A_p^p\biggl(
    \frac1{Y^p}|u_{z}|_{p}^p + \frac1{X^p}|u_{w}|_{p}^p
    \biggr) \\
  & \ge c_1^2\, 2^{-p}\, 3^{-p/2}\, A_p^p \frac{\eta^p}{X^pY^p}
    \left({X^p}|u_{z}|_{p}^p + {Y^p}|u_{w}|_{p}^p\right).
\end{align*}
}
Here, we \edit{have} used the estimates
\[
  |u_{zz}|_{p}^p + |u_{zw}|_{p}^p \ge
   \frac{A_p^p}{2^p} |u_z|_{p}^p, \quad
  |u_{zw}|_{p}^p + |u_{ww}|_{p}^p \ge
   \frac{A_p^p}{2^p} |u_w|_{p}^p,
\]
since we may apply Lemma~\ref{L2} and \Ref{similar}
to $u_z$ and $u_w$.  Combining these estimates, we obtain
\begin{align*}
  \frac{\Pi_N^p}{\Pi_D^p}
  \ge \frac{A_p^p}{2^{3\tau(p)+ 2\gamma(p)+ 3p-1} 3^{p/2}}
  \frac{\eta^p}{X^pY^p}
  = \frac{A_p^p}{2^{3\tau(p)+ 2\gamma(p)+ 3p-1} 3^{p/2} R^p}.
\end{align*}

For the case $p=\infty$, 
$\frac{|v|_{2,\infty,\KK}}{|v|_{1,\infty,\KK}} = \frac{\Pi_N}{\Pi_D}$
is written \edit{as}
\begin{gather*}
  \Pi_N := \max\biggl\{
   \max\left\{\eta|u_{xx}|_{\infty},
    |u_{xy}|_{\infty}\right\},
   \textstyle{\frac{1}{\eta}}\max\left\{\eta|u_{xy}|_{\infty},
     |u_{yy}|_{\infty}\right\} \biggr\}, \\
  \Pi_D := \max\left\{\eta |u_x|_{\infty},
                       |u_y|_{\infty} \right\}.
\end{gather*}
Then, \edit{in} exactly \edit{the} same manner we obtain
\begin{align*}
  \Pi_D  \le 2 \max\left\{X|u_z|_{\infty}, Y|u_w|_{\infty}\right\}
\end{align*}
and
{\allowdisplaybreaks
\begin{align*}
 \Pi_N & \ge \frac12 \max\left\{\eta
    \max\left\{\frac{|u_{xz}|_\infty}{Y}, \frac{|u_{xw}|_\infty}{X}
   \right\},
   \max\left\{\frac{|u_{yz}|_\infty}{Y}, \frac{|u_{yw}|_\infty}{X}
   \right\}\right\} \\
  & = \frac12 \max\left\{ \frac1{Y}
    \max\left\{\eta|u_{xz}|_\infty, |u_{yz}|_\infty\right\},
    \frac1{X}
   \max\left\{\eta|u_{xw}|_\infty, |u_{yw}|_\infty\right\}\right\} \\
    & \ge \frac{\eta}{2^2} \max\left\{
   \frac1{Y} \max\left\{\frac{|u_{zz}|_\infty}{Y},
   \frac{|u_{zw}|_\infty}{X}\right\},
   \frac1{X}\max\left\{ \frac{|u_{zw}|_\infty}{Y},
   \frac{|u_{ww}|_\infty}{X}\right\}
   \right\} \\
    & \ge \frac{\eta}{2^2 3^{1/2}} \max\left\{
   \frac1{Y}\max\left\{|u_{zz}|_\infty,|u_{zw}|_\infty \right\},
   \frac1{X}\max\left\{|u_{zw}|_\infty,|u_{ww}|_\infty\right\}
   \right\} \\
    & \ge \frac{\eta A_\infty}{2^{3}3^{1/2}} \max\left\{
   \frac1{Y} |u_{z}|_{\infty}, \frac1{X} |u_{w}|_{\infty}
   \right\} 
%
 = \frac{ A_\infty}{2^{3}3^{1/2}R}
     \max\{X |u_{z}|_{\infty}, Y |u_{w}|_{\infty}\}.
\end{align*}
}
Combining these estimates\edit{,} we conclude \edit{that}
\begin{align*}
  \frac{\Pi_N}{\Pi_D} \ge \frac{A_\infty}{2^43^{1/2}R}.
\end{align*}

\begin{lemma}\label{L7}
Let $K$ be an arbitrary triangle whose circumradius is $R$.
Suppose that the longest edge of $K$ is parallel to \edit{the} $x$-axis
(or to \edit{the} $y$-axis) and of length $2$.  \edit{Then}
\[
 B_p(K) := \inf_{v \in \T_p(K)}\frac{|v|_{2,p,K}}{|v|_{1,p,K}}
  \ge \frac{A_p}{2^{\phi(p)}3^{1/2}R}, \quad
  \phi(p) := \begin{cases}
   3/2 + 2/p & 1 \le p \le 2 \\
   4 - 3/p   & 2 \le p \le \infty
   \end{cases}.
\]
\end{lemma}

Next, we try to estimate
$\inf_{v \in \T_p(\KK)}\frac{|v|_{2,p,\KK}}{|v|_{0,p,\KK}}$.
Let $1 \le p < \infty$.  Take any function $v \in \T_p(\KK)$,
and define $u := v \circ F_{1,\eta} \in \T_p(\Kst)$.
From \Ref{trans2}, we need to estimate
\begin{align*}
  \frac{|v|_{2,p,\KK}^p}{|v|_{0,p,\KK}^p}
   = \frac{{\eta^p} |u_{xx}|_{p}^p
             + \frac{1}{\eta^p} |u_{yy}|_{p}^p
             + 2 |u_{xy}|_{p}^p}
        {\eta^p |u|_{p}^p}
    = \frac{\Pi_N^p}{\eta^p |u|_{p}^p}\edit{.}
\end{align*}
Using \Ref{eta2}, we see that
\begin{align*}
   \Pi_N^p & \ge c_1^2 \eta^p \biggl(
   \frac{|u_{zz}|_{p}^p}{Y^{2p}} +
   \frac{2|u_{zw}|_{p}^p}{X^pY^p} +
   \frac{|u_{ww}|_{p}^p}{X^{2p}} \biggr) \\
   & = c_1^2\frac{\eta^{3p}}{X^{2p}Y^{2p}} \biggl(
    \frac{X^{2p}}{\eta^{2p}}|u_{zz}|_{p}^p +
   2\frac{X^pY^p}{\eta^{2p}}|u_{zw}|_{p}^p +
   \frac{Y^{2p}}{\eta^{2p}}|u_{ww}|_{p}^p \biggr) \\
   & \ge c_1^2\frac{\eta^{p}}{3^pR^{2p}} \left(
    |u_{zz}|_{p}^p + 2|u_{zw}|_{p}^p + |u_{ww}|_{p}^p \right).
\end{align*}
Therefore, applying Lemma~\ref{L5} we obtain
\begin{align*}
  \frac{|v|_{2,p,\KK}^p}{|v|_{0,p,\KK}^p}
     = \frac{\Pi_N^p}{\eta^p |u|_{p}^p}
  & \ge \frac{c_1^2}{3^pR^{2p}} 
    \frac{ \left(|u_{zz}|_{p}^p + 2|u_{zw}|_{p}^p +
   |u_{ww}|_{p}^p \right)}{|u|_{p}^p} \\
  & \ge \frac{c_1^2}{3^pR^{2p}}
    \frac{D_p^p}{2^{2p}}
   = \frac{D_p^p}{2^{2\gamma(p)+2\tau(p)+3p}3^pR^{2p}}.
\end{align*}
The proof of the case $p=\infty$ is very similar.

\begin{lemma}\label{L8}
Let $K$ be an arbitrary triangle whose circumradius is $R$.
Suppose that the longest edge of $K$ is parallel to \edit{the} $x$-axis
(or to \edit{the} $y$-axis) and of length $2$. Then
\[
 D_p(K) := \inf_{v \in \T_p(K)}\frac{|v|_{2,p,K}}{|v|_{0,p,K}}
  \ge \frac{D_p}{2^{\mu(p)}3R^2}, \quad
  \mu(p) := \begin{cases}
   2 + 2/p & 1 \le p \le 2 \\
   4 - 2/p   & 2 \le p \le \infty
   \end{cases}.
\]
\end{lemma}

When we apply Lemmas~\ref{L7} and \ref{L8} to an arbitrary triangle
$K$, we have to notice an orthogonal matrix (or a rotation) 
$\varphi$ may change the Sobolev norms.
\edit{
More precisely, the Sobolev norms $|v\circ\varphi|_{m,p,K}$,
$m = 0,1,2$ of $v \in W^{m,p}(\varphi(K))$ are different from
$|v|_{m,p,\varphi(K)}$ in general, and are estimated as
\begin{gather*}
   2^{-m(\tau(p)+\gamma(p))/p}|v|_{m,p,\varphi(K)} \le
   |v\circ\varphi|_{m,p,K} \le 
  2^{m(\tau(p)+\gamma(p))/p}|v|_{m,p,\varphi(K)},
  \quad p \in [1,\infty),
\end{gather*}
and
\vspace{-10pt}
\begin{gather*}
     2^{-m/2}|v|_{m,\infty,\varphi(K)} \le
   |v\circ\varphi|_{m,\infty,K} \le 
  2^{m/2}|v|_{m,\infty,\varphi(K)}.
\end{gather*}
}
Gathering Lemma~\ref{L7}, \ref{L8} and \Ref{similar}, 
we have obtained the following theorems:

\begin{theorem}
Let $K$ be an arbitrary triangle whose circurmradius is $R$.
Then for any $w \in \T_p(K)$, $1 \le p \le \infty$,
there exist constants $E_1(p)$ and $E_2(p)$ \edit{that depend} only on $p$
such that the following \edit{is true}:
\begin{gather*}
  \|w\|_{1,p,K} \le M_p R |w|_{2,p,K}, \qquad
    M_p := (E_1(p)^p + E_2(p)^pR^p)^{1/p}.
\end{gather*}
\end{theorem}

\begin{theorem}
Let $K$ be an arbitrary triangle whose circurmradius is $R$.
Then the following estimate holds$:$
\begin{align*}
  \|w - {I}_h w\|_{1,p,K} \le M_p R |w|_{2,p,K}, \qquad
   \forall w \in W^{2,p}(K), \quad 1 \le p \le \infty.
\end{align*}
\end{theorem}


\section{Concluding remarks}
In this paper, \edit{we have proved} the circumradius condition for
triangular elements in $\R^2$ \edit{using a} Babu\v{s}ka-Aziz type
technique\edit{,} without validated numerical computation.  
Since the error analysis of $\Ih$ is very important, generalizations
of Kobayashi's formula and/or the circumradius condition are 
required.  Some of the \edit{unanswered}
questions \edit{on which to focus subsequent research are:}
\begin{itemize}
\item In nonlinear finite element error analysis, the 
\textbf{inverse inequality} plays an important role
\cite[pp139--143]{C}. It is interesting
to consider whether or not we are able to obtain a condition similar
to the circumradius condition for the inverse inequality.
\item \edit{Does} the circumradius condition hold on three dimensional
tetrahedrons\edit{?} 
Unfortunately, one of the authors, Kobayashi, has already found a
counter example which shows that the circumradius condition does not
hold on tetrahedrons. \edit{A} very interesting problem is to find an
essential condition\edit{, similar to the circumradius condition,
which can be used} for error estimate on tetrahedrons.
\item In \cite{HKK}, Hannukainen-Korotov-K\v{r}\'{i}\v{z}ek show that
\textit{the maximum angle condition is not necessary for convergence of
the finite element method} by showing simple examples.
In their examples, \edit{the}
circumradii of triangles are very close to constant\edit{s}
while $h \to 0$.  Thus, the circumradius condition cannot explain the
convergence of the finite element solutions in \cite{HKK}.
Therefore, the question \edit{remains}
\textit{``what is the essential \edit{triangulation} condition
for the convergence of finite element solutions\edit{?}''}.
\edit{This is a very important question, which we wish to answer.}
\end{itemize}

\vspace{-0.25cm}

\begin{acknowledgements}
The first author is supported by Inamori Foundation and
 JSPS Grant-in-Aid for Young Scientists (B) 22740059.
The second author is partially supported by JSPS
Gran\edit{t}-in-Aid for
Scientific Research (C) 22440139 and Gran\edit{t}-in-Aid for
Scientific Research (B) 23340023.
The authors thank the anonymous referee(s) for many
valuable comments and suggestions.
\end{acknowledgements}

\vspace{-0.25cm}


\begin{thebibliography}{}
%
%
\bibitem{BA}
{\sc Babu\v{s}ka, I., Aziz,A.K.:} 
On the angle condition in the finite element method,
SIAM J.\ Numer.\ Anal.\ \textbf{13}, 214--226 (1976)

\bibitem{C}
{\sc Ciarlet, P.G.:}
\textsl{The Finite Element Methods for Elliptic Problems,}
North Holland, 1978, reprint by SIAM 2008.

\bibitem{HKK}
{\sc Hannukainen, A.,Korotov, S. K\v{r}\'{i}\v{z}ek, M.:}
The maximum angle condition is not necessary for convergence of
the finite element method, 
Numer.\ Math., \textbf{120}, 79--88 (2011)

\bibitem{J}
{\sc Jamet, P.:}
 Estimations d'erreur pour des elements finis droits
presque degeneres, R.A.I.R.O.\ Anal.\ Numer., \textbf{10}, 43--61
(1976)

\bibitem{K1}
{\sc Kobayashi, K.:}
On the interpolation constants over triangular elements
(in Japanese),  RIMS Kokyuroku, \textbf{1733}, 58-77 (2011)

\bibitem{K2}
{\sc Kobayashi, K.:}
Remarkable upper bounds for the interpolation error
constants on the triangles, \textsl{in preparation}

\bibitem{KJF}
{\sc Kufner, A., John, O., Fu\v{c}\'{i}k, S.:}
\textsl{Function Spaces,} Noordhoff International Publishing, 1977

\bibitem{LK}
{\sc Liu, X., Kikuchi, F.:}
Analysis and estimation of error constants for $P_0$ and $P_1$
interpolations over triangular finite elements,
J.\ Math.\ Sci.\ Univ.\ Tokyo, \textbf{17}, 27--78 (2010)

\bibitem{Z}
{\sc Zl\'amal, M.:}
On the finite element method,
Numer.\ Math.\, \textbf{12}, 394--409 (1968)
\end{thebibliography}


\end{document}